\documentclass[11pt]{amsart}
\usepackage{amsmath,amsfonts,amstext,amssymb,amscd,bezier,amsthm} 
\newtheorem{theorem}{Theorem}
 \newtheorem{corollary}{Corollary}
 \newtheorem{lemma}{Lemma}
 \newtheorem{proposition}{Proposition}
 \theoremstyle{definition}
 \newtheorem{definition}{Definition}
 \theoremstyle{remark}
 \newtheorem{remark}{Remark}
 \newtheorem{example}{Example}
 \numberwithin{equation}{section}

\DeclareMathOperator{\Coin}{Coin}
\newcommand{\Z}{\mathbb{Z}} 

\begin{document}

\title{Computation of Nielsen and Reidemeister coincidence numbers for multiple maps}

\author{Tha\'is F. M. Monis}
\address{Universidade Estadual Paulista (UNESP) \\ Instituto de Geoci\^encias e Ci\^encias Exatas  (IGCE) \\
Av 24A, 1515, Bela Vista\\
CEP 13506-900, Rio Claro-SP, Brazil.}
\email{thais.monis@unesp.br}
\author{Peter Wong}
\address{Department of Mathematics, Bates College, Lewiston, ME 04240, U.S.A.}
\email{pwong@bates.edu}

\thanks{This work is supported by FAPESP Grant 2018/03550-5}

\begin{abstract}
Let $f_1,...,f_k:M\to N$ be maps between closed manifolds, $N(f_1,...,f_k)$ and $R(f_1,...,f_k)$ be the Nielsen and the Reideimeister coincidence numbers respectively. In this note, we relate $R(f_1,...,f_k)$ with $R(f_1,f_2),...,R(f_1,f_k)$. When $N$ is a torus or a nilmanifold, we compute $R(f_1,...,f_k)$ which, in these cases, is equal to $N(f_1,...,f_k)$.
\end{abstract}
\date{\today}
\maketitle

\section{Introduction}

A central problem in classical Nielsen coincidence theory is the computation of the Nielsen coincidence number $N(f,g)$ for two maps $f,g: M\to N$ between two closed orientable manifolds of the same dimension. Moreover, the classical Reidemeister number $R(f,g)$ is an upper bound for $N(f,g)$.

In \cite{Christo}, P. C. Staecker stablished a theory for coincidences of multiple maps called {\it Nielsen equalizer theory}: given $k$ maps $f_1 , \ldots , f_k : M^{(k-1)n} \to N^n$  between compact manifolds of dimension $(k-1)n$ and $n$, respectively,  a Nielsen number $N(f_1, \ldots , f_k)$ is defined such that it is a homotopy invariant and a lower bound for the cardinality of the sets $$\text{Coin}(f'_1, \ldots , f'_k)=\{x \in M \ | \ f_1'(x)=\cdots = f_k'(x)\},$$ where $f_i'$ is homotopic to $f_i$. The approach of Staecker to define the essentiality of a coincidence class is via a coincidence index or semi-index.     Independently, Biasi, Libardi \& Monis (\cite{Biasi}) introduced a Lefschetz coincidence class for $k$ maps $f_1, \ldots , f_k: X \to N^n$, $L(f_1, \ldots , f_k) \in H^{(k-1)n}(X; \Z)$, where $N$ is a closed connected orientable $n$-manifold, with the property that   if $L(f_1, \ldots , f_k)$ is a non-trivial element of $H^{(k-1)n}(X; \Z)$ then there is $x\in X$ such that $f_1(x)=\cdots = f_k(x)$.  In \cite{MonSpiez}, Monis \& Spie\. z generalized this Lefschetz class to the case where $N$ is not necessarily orientable by using twisted coefficients and in \cite{MonWon}, Monis \& Wong determined the obstruction class to deform the maps  $f_1 , \ldots , f_k : M^{(k-1)n} \to N^n$ to be coincidence free.  
       
In this paper, we focus on the computation of $R(f_1,...,f_k)$ in terms of the Reidemeister coincidence numbers $R(f_1,f_2), R(f_1,f_3),...,R(f_1,f_k)$. In particular, we study the cases when the target manifold is either a torus or a nilmanifold.

We should point out that in the case of positive codimension, other Nielsen type invariants via normal bordism have been introduced and Wecken type theorems have been proven. However, these invariants in general are not readily computable. In this work, the Nielsen coincidence number $N(F,G)$ (and $N(f_1,...,f_k)$ for multiple maps) we use is the geometric invariant that is index free. Thus, the calculation is carried out at the fundamental group level.    

We thank the anonymous referee for his/her comments, the question related to the divisibility result in section 5, and the suggestion for extending Theorem \ref{jiang-type-estimate}.     
            
\section{Nielsen coincidence number for multiple maps}  

In this section, we will define a geometric Nielsen coincidence number for multiple maps following the approach of \cite{Brooks1}. Throughout this section, $X$ is a connected finite CW-complex and $Y$  is a manifold without boundary.

\begin{definition} Let $f_1, f_2, \ldots , f_k: X \to Y$ be maps. Two points $x_0, x_1 \in \text{Coin}(f_1, \ldots , f_k)$ are called Nielsen equivalent as coincidences with respect to $f_1, \ldots , f_k$ if there is a path $\gamma: [0,1] \to X$ such that $\gamma(0)=x_0$, $\gamma (1)=x_1$ and $f_1(\gamma)$ is homotopic to $f_j (\gamma)$ relative to the endpoints, $j=2, \ldots ,k$. Such relation defines an equivalence relation on $\text{Coin}(f_1, \ldots , f_k)$ and each equivalence class is called a  Nielsen coincidence class of $(f_1, \ldots , f_k)$.

\end{definition}


\begin{definition} Let $f_1, \ldots , f_k : X \to Y$ be maps and  $\{f_1^t\}, \{f_2^t\}, \ldots , \{f_k^t\}$ homotopies of $f_1=f_1^0$, $f_2=f_2^0, \ldots , f_k=f_k^0$, respectively. Let $\mathcal{F}_1$ be a Nielsen coincidence class of $(f_1, \ldots , f_k)$
and  $\mathcal{F}_2$  a Nielsen coincidence class of $(f_1^1, \ldots , f_k^1)$. We say that $\mathcal{F}_1$ and $\mathcal{F}_2$ are $\{f_1^t\}, \{f_2^t\}, \ldots , \{f_k^t\}$-related if there are $x\in \mathcal{F}_1$, $x' \in \mathcal{F}_2$ and a path $\gamma: [0,1] \to X$ with $\gamma(0)=x$, $\gamma(1)=x'$ such that the paths $t \mapsto f_1^t(\gamma(t))$ and $t \mapsto f_i^t(\gamma(t))$ are homotopic
relative to the endpoints, $i=2, \ldots , k$.

\end{definition}

\begin{definition} A Nielsen coincidence class $\mathcal{F}$ of $(f_1, \ldots , f_k)$ is said to be (geometric) essential if for any homotopies $\{f_1^t\}, \{f_2^t\}, \ldots , \{f_k^t\}$ of $f_1=f_1^0$, $f_2=f_2^0, \ldots , f_k=f_k^0$,  there is a Nielsen coincidence class of $(f_1^1, \ldots , f_k^1)$ to which it is  $\{f_1^t\}, \{f_2^t\}, \ldots , \{f_k^t\}$-related. Otherwise, it is called (geometric) inessential. 

The (geometric) Nielsen coincidence number of $(f_1, \ldots , f_k)$, $N(f_1, \ldots , f_k)$, is the number of (geometric) essential Nielsen coincidence classes of $(f_1, \ldots , f_k)$.

\end{definition}

A clear property with the above definition is that, for any permutation $\sigma \in S_k$ on $\{1,...,k\}$, 
$N(f_1,...,f_k)=N(f_{\sigma(1)}, ..., f_{\sigma(k)})$.

\

%
%
%

\begin{remark} Let $\mathcal{F}$ be an inessential Nielsen coincidence class of $(f_1, \ldots , f_k)$. Roughly speaking, $\mathcal{F}$ {\it disappears} under some homotopy. More precisely, it means that there are homotopies  $\{f_1^t\}, \{f_2^t\}, \ldots , \{f_k^t\}$ of $f_1=f_1^0$, $f_2=f_2^0, \ldots , f_k=f_k^0$, respectively, such that every $x\in \mathcal{F}$ is not $\{f_1^t\}, \{f_2^t\}, \ldots , \{f_k^t\}$-related to any coincidence $x'\in \text{Coin}(f_1^1, \ldots , f_k^1)$.
On the other hand, $\mathcal{F}$ is essential if and only if for any homotopies $\{f_1^t\}, \{f_2^t\}, \ldots , \{f_k^t\}$ of $f_1=f_1^0$, $f_2=f_2^0, \ldots , f_k=f_k^0$, respectively, there exists $x' \in \text{Coin}(f_1^1, \ldots , f_k^1)$ and a path $\gamma: [0,1] \to X$ with $\gamma(0)=x \in \mathcal{F}$, $\gamma(1)=x'$ such that $t \mapsto f_1^t(\gamma(t))$ and $t \mapsto f_i^t(\gamma(t))$ are homotopic paths relative to the endpoints, $i=2, \ldots , k$.

\end{remark}

\begin{remark} Let $F, G : X \to Y^{k-1}$ be defined by $F=(f_1, f_1, \ldots , f_1)$ and $G=(f_2, f_3, \ldots , f_k)$. It is not difficult to see that $\text{Coin}(F,G) = \text{Coin}(f_1, \ldots , f_k)$ and that the Nielsen coincidence classes of $(f_1, \ldots , f_k)$ coincide with the Nielsen coincidence classes of $(F,G)$. In the next result, we prove that $\mathcal{F}$ is essential as a coincidence class of $(f_1, \ldots , f_k)$ if and only if it is essential as a coincidence class of $(F,G)$. Therefore,  $N(f_1, \ldots , f_k) = N(F, G)$.
\end{remark}

\begin{proposition}\label{formula1} $N(f_1, \ldots , f_k) = N(F,G)$.

\end{proposition}
\begin{proof} Let $\mathcal{F}$ be a Nielsen coincidence class of $(f_1, \ldots , f_k)$ and, therefore, a Nielsen coincidence class of $(F, G)$. We will show that $\mathcal{F}$ is essential as a coincidence class of $(f_1, \ldots , f_k)$ if and only if it is essential as a coincidence class of $(F,G)$.

Suppose $\mathcal{F}$ is essential as a coincidence class of $(F,G)$. Let $\{f_1^t\}, \{f_2^t\}, \ldots ,$ $\{f_k^t\}$ be homotopies of $f_1=f_1^0$, $f_2=f_2^0, \ldots , f_k=f_k^0$, respectively. Then, $\{(f_1^t, \ldots , f_1^t)\}$, $\{(f_2^t, \ldots , f_k^t)\}$ are homotopies of $F$ and $G$, respectively. Since $\mathcal{F}$ is essential as a coincidence class of $(F,G)$,
there exists $$x' \in \text{Coin}((f_1^1, \ldots , f_1^1), (f_2^1, \ldots , f_k^1)) = \text{Coin}(f_1^1, \ldots , f_k^1)$$ and a path $\gamma: [0,1] \to X$ with $\gamma(0)=x \in \mathcal{F}$, $\gamma(1)=x'$ such that $t \mapsto (f_1^t(\gamma(t)), \ldots , f_1^t(\gamma(t)))$ and $t \mapsto (f_2^t(\gamma(t)), \ldots , f_k^t(\gamma(t)))$ are homotopic paths relative to the endpoints. Therefore, $t \mapsto f_1^t(\gamma(t))$ and $t \mapsto f_i^t(\gamma(t))$ are homotopic paths relative to the endpoints, $i=2, \ldots , k$. Hence, $\mathcal{F}$ is essential as a coincidence class of $(f_1, \ldots , f_k)$.

Now, suppose $\mathcal{F}$ is essential as a coincidence class of $(f_1, \ldots , f_k)$. Since $Y$ is a manifold, it follows from \cite[Corollary 2]{Brooks} that $\mathcal{F}$ is essential as a coincidence class of $(F,G)$ iff it is essential with respect to homotopies of the form $(F, G^t)$, in other words, the homotopy of $F$ can be kept constant. Since $\mathcal{F}$ is essential as a coincidence class of $(f_1, \ldots , f_k)$, under the homotopies $\{f_1\}, \{f_2^t\}, \ldots , \{f_k^t\}$, where $\bar G^t=(f_2^t, \ldots , f_k^t)$, there exists a path $\gamma:[0,1]\to X$ with $\gamma(0)=x$ and $\gamma(1)\in \text{Coin}(f_1,f_2^1,...,f_k^1)=\text{Coin}(F,\bar G^1)$ such that the path $\{f_1(\gamma)(t)\}$ and the path $\{f_i^t(\gamma)(t)\}$ are homotopic relative to the endpoints for $i=2,...,k$. Thus, by \cite[Corollary 2]{Brooks}, we conclude that $\mathcal{F}$ is essential as a coincidence class of $(F,G)$.  
\end{proof}

\begin{remark} Since $X$ is compact and the Nielsen coincidence classes are both open and closed in $\text{Coin}(F,G)$, this index-free geometric Nielsen number $N(F,G)$ (and hence $N(f_1,...,f_k)$) is finite, is a lower bound for the number of connected components in $\text{Coin}(F,G)$ (and hence $\text{Coin}(f_1,...,f_k)$), and is a homotopy invariant.
\end{remark}

\section{Reidemeister coincidence classes for multiple maps}

In \cite{Christo}, P. C. Staecker also developed a Reidemeister-type theory for coincidences of multiple maps. Let $X$ and $Y$ be connected, locally path-connected, and semilocally simply connected spaces, and $\widetilde{X}, \widetilde{Y}$ their universal covering spaces with projection maps $p_X: \widetilde{X} \to X$ and $p_Y: Y \to \widetilde{Y}$, respectivelly. Let $f_1, \ldots , f_k: X \to Y$ be maps, $k>2$, and denote $$\Coin (f_1, \ldots , f_k) = \{ x\in X \ | \ f_1(x) = \cdots = f_k(x)\} $$ the set of coincidences of the maps $f_1, \ldots , f_k$.

\begin{theorem}[\cite{Christo}, Theorem 2.1] Let $f_1, \ldots , f_k : X \to Y$ be maps with lifts $\tilde{f}_i: \widetilde{X} \to \widetilde{Y}$ and induced homomorphisms $\phi_i : \pi_1(X) \to \pi_1(Y)$. Then:
\begin{enumerate}
\item $$\Coin(f_1, \ldots , f_k) = \bigcup_{\alpha_2, \ldots , \alpha_k \in \pi_1(Y)} p_X ( \Coin (\tilde{f}_1, \alpha_2 \tilde{f}_2, \ldots , \alpha_k \tilde{f}_k)) .$$
\item For $\alpha_i, \beta_i \in \pi_1(Y)$, the sets $$p_X ( \Coin (\tilde{f}_1, \alpha_2 \tilde{f}_2, \ldots , \alpha_k \tilde{f}_k)) \ \ \text{and} \ \ p_X ( \Coin (\tilde{f}_1, \beta_2 \tilde{f}_2, \ldots , \beta_k \tilde{f}_k))  $$ are disjoint or equal. 
\item The above sets are equal if and only if there is some $z\in \pi_1(X)$ with $\beta_i = \phi_1(z) \alpha_i \phi_i(z)^{-1}$ for all $i$.

\end{enumerate}

\end{theorem}

\

The above theorem gives rise to the definition of the set of Reidemeister classes and the Reidemeister number for $\phi_1, \ldots , \phi_k$ : given $(\alpha_2, \ldots , \alpha_k)$ and $(\beta_2, \ldots , \beta_k)$ in $\pi_1(Y)^{k-1}$, 
$(\alpha_2, \ldots , \alpha_k) \sim (\beta_2, \ldots , \beta_k)$ if and only if $\beta_i = \phi_1(z) \alpha_i \phi_i(z)^{-1}$, $i=2, \ldots , k$, for some $z \in \pi_1(X)$. The quotient of $\pi_1(Y)^{k-1}$ by such equivalence relation is denoted by $\mathcal{R}(\phi_1, \ldots , \phi_k) = \pi_1(Y)^{k-1}/\sim$ and each class is called a Reidemeister class for $\phi_1, \ldots , \phi_k$. The cardinality of $\mathcal{R}(\phi_1, \ldots , \phi_k)$ is the Reidemeister number for $\phi_1, \ldots , \phi_k$ and it is denoted by $R(\phi_1, \ldots , \phi_k)$.

\begin{remark}\label{remark3} Note that $\mathcal{R}(\phi_1, \ldots , \phi_k)= \mathcal{R}((\phi_1, \ldots , \phi_1), (\phi_2, \ldots , \phi_k))$ the usual Reidemeister set for the two homomorphisms $(\phi_1, \ldots , \phi_1), (\phi_2, \ldots , \phi_k): \pi_1(X) \to \pi_1 (Y)^{k-1}$ induced by the maps $(f_1, \ldots , f_1), (f_2, \ldots , f_k): X \to Y^{k-1}$, respectivelly.

\end{remark}

\begin{proposition} \label{invariance}
For any permutation $\sigma \in S_k$ on $\{1,...,k\}$, 
$$ R(f_1,...,f_k)=R(f_{\sigma(1)}, ..., f_{\sigma(k)}). $$
\end{proposition}
\begin{proof} Since a permutation is a product of transpositions, it is sufficient to prove the result for the transpositions.

\noindent \underline{Case 1}: Suppose $\sigma = (i j)$ with $\{ i, j\} \subset \{2,3, \ldots k\}$. Then, it is not difficult to see that $$\mathcal{R}(f_1, \ldots, f_k ) \to \mathcal{R}(f_{\sigma(1)}, ..., f_{\sigma(k)}) $$ 
$$[(\alpha_2, \ldots , \alpha_k)] \to [(\alpha_{\sigma(2)}, \ldots , \alpha_{\sigma(k)})] $$ is a well-defined bijection between the Reidemeister classes sets $\mathcal{R}(f_1, \ldots, f_k )$ and $\mathcal{R}(f_{\sigma(1)}, ..., f_{\sigma(k)}) $. Therefore, $R(f_1,...,f_k)=R(f_{\sigma(1)}, ..., f_{\sigma(k)})$.

\noindent \underline{Case 2}: Suppose $\sigma = (1 i)$,  for some $i\in \{2, \ldots , k\}$. Then, $$\mathcal{R}(f_1, \ldots, f_k ) \to \mathcal{R}(f_{\sigma(1)}, ..., f_{\sigma(k)}) $$ 
$$[(\alpha_2, \ldots , \alpha_k)] \to [(\alpha_i^{-1} \alpha_2, \ldots ,  \alpha_i^{-1} \alpha_{i-1}, \alpha_i^{-1},     \alpha_i^{-1} \alpha_{i+1}, \ldots ,  \alpha_i^{-1} \alpha_k)] $$ is a well-defined bijection between the Reidemeister classes sets $\mathcal{R}(f_1, \ldots, f_k )$ and $\mathcal{R}(f_{\sigma(1)}, ..., f_{\sigma(k)}) $. Therefore, $R(f_1,...,f_k)=R(f_{\sigma(1)}, ..., f_{\sigma(k)})$.\end{proof}

\section{Jiang type results for coincidences of multiple maps}

In Nielsen coincidence theory, for a large class of spaces known as Jiang-type spaces, either $N(f,g)=0$ or $N(f,g)=R(f,g)$. In this section, we also obtain similar results for multiple maps.

\begin{theorem} \label{Nil-Jiang}
Let $X$ and $N$ be compact nilmanifolds with $\dim X \geq (k-1) \dim N$, $k \geq 2$. For any maps $f_1,..., f_k$, either
$N(f_1,...,f_k)= 0$ or $N(f_1,...,f_k)= R(f_1,...,f_k)$.
\end{theorem}
\begin{proof}  By Proposition \ref{formula1}, $N(f_1,...,f_k)=N(F,G)$, where $F, G : X \to N^{k-1}$, $F=(f_1, \ldots , f_1)$ and $G=(f_2, \ldots , f_k)$. Moreover, $R(f_1,...,f_k)=R(F,G)$ (see Remark \ref{remark3}).
From \cite[Theorem 4.2]{daci-peter2}, either $N(F,G)=0$ or $N(F,G)=R(F,G)$. Therefore, either
$N(f_1,...,f_k)= 0$ or $N(f_1,...,f_k)= R(f_1,...,f_k)$.
\end{proof}                    
                    
When one of the maps is the constant map, we obtain similar results as in the classical root case.                    
                    
\begin{theorem} \label{root} Let $X$ be a topological space, $Y$ a topological manifold, $c\in N$ and denote by $\bar c$ the constant map at $c$. For any maps $f_2,..., f_k: X \to N$, $k\ge 2$, either
$N(\bar c,f_2, ...,f_k)= 0$ or $N(\bar c, f_2,...,f_k)= R(\bar c, f_2,...,f_k)$.
\end{theorem}
\begin{proof} It follows from \cite[Corollary 2]{Brooks2},  since $$N(\bar c,f_2, ...,f_k)=N((\bar{c}, \ldots , \bar{c}), (f_2, \ldots , f_k))$$ and   $$R(\bar c,f_2, ...,f_k)=R((\bar{c}, \ldots , \bar{c}), (f_2, \ldots , f_k)).$$
\end{proof}

Analogous to \cite[Theorem 6.3]{MonWon}, we have the following

\begin{theorem}\label{jiang-type}
Let $f_1,\ldots,f_k:M\to N$ from a closed connected $(k-1)n$-manifold $M$ to a closed connected orientable $n$-manifold $N$. Suppose $N$ is a Jiang space; a nilmanifold, an orientable coset space $G/K$ of a compact connected Lie group $G$ by a closed subgroup $K$; or a $\mathcal C$-nilpotent space whose fundamental group has a finite index center where $\mathcal C$ is the class of finite groups. Then $L(f_1,\ldots,f_k)=0 \Rightarrow N(f_1,...,f_k)=0$ and $L(f_1,\ldots,f_k)\ne 0 \Rightarrow N(f_1,...,f_k)=R(f_1,...,f_k)$.
\end{theorem}

\begin{remark} Every compact nilmanifold is the total space of a principal $S^1$-bundle over another nilmanifold. Using the upper central series, there is a sequence of $S^1$ fibrations such that an $2n$-dimensional compact nilmanifold $M$ can fiber over an $n$-dimensional nilmanifold $N$ (see e.g. \cite{w2}). Consider such a fibration $p:M\to N$ and a point $c\in N$. If $M$ is not symplectic then for any map $f:M\to N$, it follows from \cite[Example 7.3]{MonWon} that the Lefschetz class $L_2(p,\bar c, f)=0$. It follows that $N(p,\bar c,f)=0$ so, by \cite{daci-peter}, $R(p,\bar c,f)=\infty$ while $R(p,\bar c)<\infty$ because $p$ and $\bar c$ cannot be made coincidence free.
The existence of non-symplectic nilmanifolds is equivalent to the existence of symplectic structure on nilpotent Lie algebras (see e.g., \cite{benson-gordon,salamon,delBarco}).
\end{remark}

It was shown in \cite{MonSpiez} that the Lefschetz class $L(f_1,...,f_k)$ is related to the classes $L(f_1,f_2)$, $L(f_1,f_3)$, $\ldots$, $L(f_1,f_k)$. Indeed, the obstruction $o(f_1,...,f_k)$ to deforming the $k$-maps to be coincidence free is the product $o(f_1,f_2) \cup \cdots \cup o(f_1,f_k)$ (\cite[Theorem 5.4]{MonWon}). Next, we investigate the relationship between $R(f_1,...,f_k)$ and the product $R(f_1,f_2)\cdot R(f_1,f_3) \cdots R(f_1,f_k)$ when these quantities are finite.

In general, there is a well-defined natural surjection
\begin{equation}
\Psi: \mathcal{R}(\phi_1, \ldots , \phi_k) \to \mathcal{R}(\phi_1, \phi_2) \times \cdots \times \mathcal{R}(\phi_1, \phi_k)
\end{equation}
given by 
\begin{equation}
\Psi \left([(\alpha_2, \ldots , \alpha_k)]\right) = ([\alpha_2], \ldots , [\alpha_k]).
\end{equation}

Therefore, if $ R(\phi_1, \ldots , \phi_k)< \infty$ then $R(\phi_1, \phi_j) < \infty$, $j=2, \ldots , k$ and
\begin{equation}\label{R-ineq}
R(\phi_1, \ldots , \phi_k) \ge R(\phi_1,\phi_2)  \cdots  R(\phi_1,\phi_k).
\end{equation}

\section{Coincidences of maps into a torus}

Let $ \varphi, \psi : G \to A$ be group homomorphisms, where $A$ is an abelian group. In this case, the relation that determines the Reidemeister classes with respect to the pair $(\varphi , \psi)$ can be written as follows: given $\alpha , \beta \in A$, $\alpha \sim \beta$ if and only if $\beta - \alpha = \varphi(g) - \psi (g)$ for some $g \in G$, that is, $\mathcal{R}(\varphi, \psi) = \text{coker}(\varphi - \psi)$. Therefore, $R(\varphi, \psi) = \# \text{coker}(\varphi - \psi)$.

Let $\varphi,  \psi: \Z^m \to \Z^n$ be homomorphisms, $\{\alpha_1, \ldots , \alpha_m\}$ generators of $\Z^m$ and $\{\beta_1, \ldots , \beta_n\}$ generators of $\Z^n$.  If $\varphi (\alpha_j) = \sum_{i=1}^n a_{ij} \beta_i $ and $\psi (\alpha_j) = \sum_{i=1}^n b_{ij} \beta_i$, the Reidemeister number $R(\varphi, \psi)$ is determined by the integral matrices $A=(a_{ij})$ and $B=(b_{ij})$: if all $n \times n$ minor of  $A-B$ is zero then $R(\varphi, \psi)= \infty$. If some $n\times n$ minor of $A-B$ is nonzero, then $R(\varphi, \psi)$ is finite and its effective computation can be done by using  Smith normal form. Indeed, let $C=A-B$,  by Smith normal form, there are matrices $S\in M_n(\Z)$ and $T \in M_m(\Z)$ with $|\det (S)| = |\det (T)| =1$ and $SCT$ is of the form $$\left(\begin{array}{ccccccc}
l_1 & 0 & 0 & \ldots & & & 0  \\
0 & l_2 & 0 & \ldots & & & 0 \\
0 & 0 & \ddots & \vdots & \vdots & & \vdots \\
\vdots & \cdots & & l_r & \ldots & & \vdots \\
& & & &0& \ldots & \\
\vdots & & & & & \ddots &  \\
0& & & \ldots & & & 0
\end{array} \right)$$

It follows that $R(\varphi, \psi)= \# \text{coker}(C) = \# \text{coker}(SCT)$ since $\text{coker}(C)$ and $\text{coker}(SCT) $ are isomorphic. If we assume $R(\varphi, \psi) < \infty$ then $r=n$ and $\# \text{coker}(SCT)=l_1 \cdot l_2 \cdot \cdots l_n$. The numbers $l_k$ are defined by the following: $l_k = \displaystyle \frac{d_k (C)}{d_{k-1}(C)}$, where $d_k(C)$ is the greatest common divisor of all $k \times k$ minors of $C$ and $d_0(C)=1$. For example, for the matrix $$C= \left(\begin{array}{ccc}
2 & 4 & 1 \\
2& 6 & 2
\end{array} \right) ,$$ $l_1(C)=1$ and $l_2(C)=2$. Therefore, $\# \text{coker}(C)=1 \cdot 2 = 2$.

In particular, when $m=n$, $R(\varphi, \psi) = \infty$ if $\det (A-B)=0$ and $R(\varphi, \psi) = |\det (A-B)|$ if $\det (A-B) \neq 0$.


\

\begin{proposition}\label{proposition3} Let $f_1, f_2, f_3 : T^2 \to S^1$ be maps and $\phi_1, \phi_2, \phi_3: \pi_1(T^2) \to \pi_1(S^1)$ the induced homomorphisms. If $R(\phi_1, \phi_2, \phi_3)< \infty$ then $R(\phi_1, \phi_2)< \infty$, $R(\phi_1, \phi_3) < \infty$ and  the product $R(\phi_1, \phi_2) \cdot R(\phi_1, \phi_3)$ divides $R(\phi_1, \phi_2, \phi_3)$.
\end{proposition}
\begin{proof} Let $\phi_i : \pi_1 (T^2) \to \pi_1 (S^1)$ be the homomorphism induced by $f_i$ and denote  $$ \phi_i \left(\begin{array}{c}
1 \\
0 
\end{array} \right)=a_i  \  \  \text{and} \  \  \phi_i \left(\begin{array}{c}
0 \\
1
\end{array} \right)=b_i .$$ For maps $g,h : T^2 \to S^1$, $R(g_{\#},h_{\#})=\#  \text{coker}(h_{\#} - g_{\#})$. Thus, 
\begin{equation}\label{equation1}
R(\phi_1, \phi_j) = \# \text{coker}(\phi_j - \phi_1) = \# \frac{\Z}{\langle a_j - a_1 , b_j - b_1 \rangle}, \  \  j=2,3.
\end{equation}
Also, note that $\mathcal{R}(\phi_1,\phi_2, \phi_3) = \mathcal{R}((\phi_1, \phi_1), (\phi_2, \phi_3))$, where $(\phi_1, \phi_1), (\phi_2, \phi_3): \pi_1(T^2) \to \pi_1(T^2)$ are the homomorphisms induced by the maps $(f_1, f_1), (f_2, f_3): T^2 \to T^2$ given by $$(f_i, f_j)(z, w) = (f_i(z), f_j(w)).$$ Moreover, for maps $h,g : T^2 \to T^2$, $R(g_{\#},h_{\#})=\#  \text{coker}(h_{\#} - g_{\#})$. Thus, 
\begin{equation}\label{equation2}
R(\phi_1, \phi_2, \phi_3) =  \# \text{coker}((\phi_2 , \phi_3) - (\phi_1 , \phi_1)) = \# \frac{\Z \oplus \Z}{\left\langle \left(\begin{array}{c}
a_2 - a_1 \\
a_3 - a_1
\end{array} \right) ,  \left(\begin{array}{c}
b_2 - b_1 \\
b_3 - b_1
\end{array} \right) \right\rangle},
\end{equation} 
If $R(\phi_1, \phi_2, \phi_3) < \infty$ then $$\det  \left(\begin{array}{cc}
a_2 - a_1 &  b_2 - b_1\\
a_3 - a_1 &  b_3 - b_1
\end{array} \right) \neq 0.$$ Thus, for each $j \in \{2,3\}$, $a_j - a_1 \neq 0$ or $b_j -b_1 \neq 0$. Therefore, 
\begin{equation}\label{equation3}
R(\phi_1, \phi_j) = \# \text{coker}(\phi_j - \phi_1) = \# \frac{\Z}{\langle a_j - a_1 , b_j - b_1 \rangle} = \gcd (a_j - a_1, b_j - b_1) < \infty, 
\end{equation} $j=2,3$.

From Smith normal form, 

\begin{equation}
 \# \frac{\Z \oplus \Z}{\left\langle \left(\begin{array}{c}
a_2 - a_1 \\
a_3 - a_1
\end{array} \right) ,  \left(\begin{array}{c}
b_2 - b_1 \\
b_3 - b_1
\end{array} \right) \right\rangle} = \left| \det  \left(\begin{array}{cc}
a_2 - a_1 &  b_2 - b_1\\
a_3 - a_1 &  b_3 - b_1
\end{array} \right) \right|  .
\end{equation}


\

Therefore, $$R(\phi_1, \phi_2, \phi_3) = \left| \det  \left(\begin{array}{cc}
a_2 - a_1 &  b_2 - b_1\\
a_3 - a_1 &  b_3 - b_1
\end{array} \right) \right| .$$

Let $d_{12}=\gcd (a_2 - a_1, b_2-b_1)= R(\phi_1, \phi_2)$, $d_{13}=\gcd (a_3 - a_1, b_3-b_1)= R(\phi_1, \phi_3)$, and $A,B,C,D \in \Z$ such that $a_2 - a_1 = A d_{12} $, $b_2-b_1 = B  d_{12} $, $a_3 - a_1 = C d_{13}$ and $b_3-b_1 = D d_{13}$. Thus, 

\begin{eqnarray*}
\det \left(\begin{array}{cc}
a_2 - a_1 &  b_2 - b_1\\
a_3 - a_1 &  b_3 - b_1
\end{array} \right) &=& d_{12}d_{13}  \det \left(\begin{array}{cc}
A &  B\\
C &  D
\end{array} \right) \\
&=& R(\phi_1, \phi_2) R(\phi_1, \phi_3)  \det \left(\begin{array}{cc}
A &  B\\
C &  D
\end{array} \right).
\end{eqnarray*}
%
%
%

 Therefore, $R(\phi_1, \phi_2, \phi_3) = R(\phi_1, \phi_2) \cdot R(\phi_1, \phi_3) \left| \det  \left(\begin{array}{cc}
A &  B \\
C &  D
\end{array} \right) \right| .$
%
%
\end{proof}

The result above can easily be generalized to $k>3$, as we show in Corollary \ref{R-divisible}. Note also that the Reidemeister number $R(\phi_1,\phi_2,\phi_3)$ is divisible by the product $R(\phi_1,\phi_2)\cdot R(\phi_1,\phi_3)$. In fact, such divisibility is valid in general for abelian groups (not necessarily torsion-free).

Let $f, g : X \to Y$ be maps and suppose $\pi_1(Y)$ abelian. Then, $\mathcal{R}(f_{\#},g_{\#})$ is a group (abelian). In fact, $$\mathcal{R}(f_{\#},g_{\#}) = \text{coker}(g_{\#} - f_{\#}) = \frac{\pi_1(Y)}{\text{im}(g_{\#} - f_{\#})}.$$

Consequently, if $\pi_1(Y)$ is an abelian group then $\mathcal{R}(\phi_1, \ldots , \phi_k)$, $\mathcal{R}(\phi_1, \phi_2)$, $\ldots$, $\mathcal{R}(\phi_1, \phi_k)$ are all (abelian) groups. Also, one can check that, in this case, $\Psi$ is a group homomorphism.    Hence, 
\begin{equation}
\frac{\mathcal{R}(\phi_1, \ldots , \phi_k)}{\text{ker}(\Psi)} \simeq \mathcal{R}(\phi_1, \phi_2) \times \cdots \times \mathcal{R}(\phi_1,  \phi_k).
\end{equation}

Therefore, if $R(\phi_1, \ldots , \phi_k) < \infty$ then 
\begin{equation}\label{R-product}
R(\phi_1, \ldots , \phi_k)=R(\phi_1, \phi_2) \cdot \cdots \cdot R(\phi_1, \phi_k) \cdot |\text{ker}(\Psi)|.
\end{equation}

\begin{corollary} \label{R-divisible}
Suppose $\pi_1(Y)$ an abelian group. If $R(\phi_1, \ldots , \phi_k) < \infty$  then $R(\phi_1, \phi_j) < \infty$, $j=2, \ldots , k$, and the product $R(\phi_1, \phi_2) \cdot \cdots \cdot R(\phi_1, \phi_k)$ divides $R(\phi_1, \ldots , \phi_k)$.
\end{corollary}

\

\begin{remark} We should point out that if $\pi_1(Y)$ is not abelian then the product $R(\phi_1, \phi_2) \cdot R(\phi_1, \phi_3)$ may not divide $R(\phi_1, \phi_2, \phi_3)$ as we show in the next example.
\end{remark}

\begin{example}
Let $Y$ be the Poincar\'e $3$-dimensional sphere, $X=Y \times Y$, $f_1=f_2=p: Y \times Y \to Y$ the projection onto the first coordinate and $f_3=\bar{c}: Y \times Y \to Y$ the constant map. The relation on $\pi_1(Y) \times \pi_1(Y)$ that defines $\mathcal{R}(p,p,\bar{c})$ is given by: $(\alpha_2, \alpha_3) \sim (\beta_2, \beta_3)$ if and only if $(\beta_2, \beta_3)=(z \alpha_2 z^{-1}, z \alpha_3)$ for some $z \in \pi_1(Y)$. So, it is not difficult to see that the number of elements in the class of $(\alpha_2, \alpha_3)$ is in 1-1 correspondence with $|\pi_1(Y)|$. Therefore, $R(\phi_1, \phi_2, \phi_3) = 120^2/120=120$. On the other hand, the relation on $\pi_1(Y)$ that defines $\mathcal{R}(p,p)$ is given by $\alpha \sim \beta$ if and only if $\beta=z \alpha z^{-1}$ for some $z \in \pi_1(Y)$, that is, $\mathcal{R}(p,p)$ is the set of conjugacy classes of $\pi_1(Y)$, and it is known that there are $9$ conjugacy classes. Hence, $R(p,p)=9$. For $\mathcal{R}(p, \bar{c})$, the relation on $\pi_1(Y)$ is given by $\alpha \sim \beta$ if and only if $\beta = z \alpha$ for some $z \in \pi_1(Y)$. Hence, $R(p, \bar{c})=1$. Now, $9 \nmid 120$.
\end{example}



\begin{remark} As suggested by the anonymous referee,  one may ask if the divisibility in Proposition \ref{proposition3} can be generalized for $k>3$ in the following sense: $\prod_{i=1}^k R(\phi_1, \ldots , \phi_{i-1}, \phi_{i+1}, \ldots , \phi_k)$ divides $R(\phi_1, \ldots , \phi_k)$. The following example shows that such divisibility does not hold in general. 
\end{remark}

\begin{example}
Let $f_1, f_2, f_3, f_4 : T^3 \to S^1$ be maps with $\phi_1={f_1}_{\#}, \phi_2={f_2}_{\#}, \phi_3={f_3}_{\#}, \phi_4={f_4}_{\#}: \Z^3 =\pi_1(T^3) \to \pi_1 (S^1)=\Z$ given by $\phi_1= \left(\begin{array}{ccc}
1 &  1 &1
\end{array} \right) $, $\phi_2= \left(\begin{array}{ccc}
3 &  5 &2
\end{array} \right) $, $\phi_3= \left(\begin{array}{ccc}
3 &  7 &3
\end{array} \right) $ and $\phi_4= \left(\begin{array}{ccc}
2 &  1 &3
\end{array} \right) $. We know that $$R(\phi_1, \phi_2, \phi_3, \phi_4) = \left| \det  \left(\begin{array}{ccc}
a_2 - a_1 &  b_2 - b_1 & c_2 - c_1 \\
a_3 - a_1 &  b_3 - b_1 & c_3 - c_1 \\
a_4 - a_1 & b_4 - b_1 & c_4 - c_1
\end{array} \right) \right| ,$$ where  $\phi_1= \left(\begin{array}{ccc}
a_1 &  b_1 &c_1
\end{array} \right) $, $\phi_2= \left(\begin{array}{ccc}
a_2 &  b_2 & c_2
\end{array} \right) $, $\phi_3= \left(\begin{array}{ccc}
a_3 &  b_3 &c_3
\end{array} \right) $ and $\phi_4= \left(\begin{array}{ccc}
a_4 &  b_4 & c_4
\end{array} \right) $. Therefore,  $$R(\phi_1, \phi_2, \phi_3, \phi_4) = \left| \det  \left(\begin{array}{ccc}
2 &  4 & 1 \\
2 &  6 & 2 \\
1 & 0 & 2
\end{array} \right) \right| =10.$$ On the other hand, $R(\phi_1, \phi_2, \phi_3)= \# \text{Coker}((\phi_2, \phi_3) - (\phi_1, \phi_1))=\# \text{Coker} \left(\begin{array}{ccc}
2 &  4 &1 \\
2 & 6 & 2
\end{array} \right) $, $R(\phi_1, \phi_2, \phi_4)= \# \text{Coker}((\phi_2, \phi_4) - (\phi_1, \phi_1))=\# \text{Coker} \left(\begin{array}{ccc}
2 &  4 &1 \\
1 & 0 & 2
\end{array} \right) $ and $R(\phi_1, \phi_3, \phi_4)= \# \text{Coker}((\phi_3, \phi_4) - (\phi_1, \phi_1))=\# \text{Coker} \left(\begin{array}{ccc}
2 &  6 &2 \\
1 & 0 & 2
\end{array} \right) $. By using Smith normal form, one can show that $R(\phi_1, \phi_2, \phi_3)=2$, $R(\phi_1, \phi_2, \phi_4)=1$  and $R(\phi_1, \phi_3, \phi_4)=2$. Therefore, $R(\phi_1, \phi_2, \phi_3) \cdot R(\phi_1, \phi_2, \phi_4) \cdot R(\phi_1, \phi_3, \phi_4) = 4$ does not divide $R(\phi_1, \phi_2, \phi_3, \phi_4)=10$.

\end{example}

\section{Coincidences of maps into a nilmanifold}

As pointed out in Remark 4 that $R(f_1,...,f_k)=R(F,G)$ where $F=(f_1,...,f_1)$ and $G=(f_2,...,f_k)$, the computation of $R(f_1,...,f_k)$ amounts to the computation of the Reidemeister number of two maps. In this section, we show how the Reidemeister coincidence number $R(f,g)$ can be computed for two maps $f,g:X\to N$ when $N$ is a nilmanifold.

From \cite{daci-peter3}, we know that for any $\tilde f,\tilde g:X\to N$, there is a nilmanifold $\hat N$ with nilpotency class $c(\pi_1(\hat N))\le c(\pi_1(N))$ and two maps $f,g:\hat N\to N$ such that, up to homotopy, $\tilde f=f\circ q$ and $\tilde g=g\circ q$ where $q:X\to \hat N$ induces a surjection on the fundamental groups. It follows that $R(f,g)=R(\tilde f,\tilde g)$. This reduces to the case where the domain is also a compact nilmanifold.

From now on, 
we let $G_1=\pi_1(\hat N)$, $G_2=\pi_1(N)$, and $c(G_1)\le c(G_2)$. Suppose $\varphi, \psi$ are the corresponding induced homomorphisms of $f,g$ respectively. 

\subsection{Central extension}\label{Central}
Let $c=c(G_2)$ be the nilpotency class of $G_2$ and $\{\gamma_i(G_j)\}$ be the lower central series of $G_j, j=1,2$. The subgroup $\gamma_{c-1}(G_j)$ is central in $G_j$ since $\{1\}=\gamma_c(G_j)=[\gamma_{c-1}(G_j),G_j]$ for $j=1,2$.

We have the following commutative diagram

\begin{equation}\label{central-exact}
\begin{CD}
    1 @>>> A_1    @>{i_1}>>  G_1 @>{p_1}>>    B_1 @>>> 1 \\
    @.     @V{\varphi'}V{\psi'}V  @V{\varphi}V{\psi}V   @V{\overline \varphi}V{\overline \psi}V @.\\
    1 @>>> A_2    @>{i_2}>>  G_2 @>{p_2}>>    B_2 @>>> 1 
 \end{CD}
\end{equation}
where $\varphi' = \varphi|_{A_1}$, $\psi' = \psi|_{A_1}, A_j=\gamma_{c-1}(G_j)$ so that $B_j=G_j/\gamma_{c-1}(G_j)$ for $j=1,2$.

It follows from \cite[Remark 2.1 and Theorem 2.1]{daci-peter6} that
\begin{equation}\label{R-central}
R(\varphi, \psi)=\#\widehat i_2(\mathcal R(\varphi',\psi'))\cdot R(\overline \varphi, \overline \psi).
\end{equation}
Here, the function $\widehat i_2$ is induced by the inclusion $i_2$ as in \eqref{central-exact}.

\begin{lemma}\label{central-lemma}
Given the commutative diagram \eqref{central-exact}, if ${\rm rk}(Coin(\varphi',\psi'))={\rm rk}(A_1)-{\rm rk}(A_2)$ then
$$
R(\varphi, \psi)\cdot |{\rm Im}\delta |=R(\varphi',\psi')\cdot R(\overline \varphi, \overline \psi)
$$
where $\delta: Coin(\overline \varphi, \overline \psi)\to \mathcal R(\varphi',\psi')$ is given by $\delta(\overline \theta)=[\psi(\theta)\varphi(\theta)^{-1}]$ where $p_1(\theta)=\overline \theta$.
\end{lemma}
\begin{proof}
It follows from \cite{heath} that there is an $8$-term exact sequence
$$
1\to Coin(\varphi',\psi') \to Coin(\varphi,\psi) \stackrel{p_1}{\to} Coin(\overline \varphi,\overline \psi) \stackrel{\delta}{\to} \mathcal R(\varphi',\psi') \stackrel{\widehat i_2}{\to} \mathcal R(\varphi,\psi) \stackrel{\widehat p_2}{\to} \mathcal R(\overline \varphi,\overline \psi) \to 1
$$
where the first four terms are groups and the other four terms are sets in general.

Since $A_2$ is central, it follows that $\mathcal R(\varphi',\psi')$ is an abelian group and $\delta$ is a group homomorphism. Since  ${\rm rk}(Coin(\varphi',\psi'))={\rm rk}(A_1)-{\rm rk}(A_2)$, it follows from \cite{w1} that $\mathcal R(\varphi',\psi')$ is finite. It is straightforward to verify that the map
$$
\mathcal R(\varphi',\psi')/{\rm Im}\delta \to \widehat i_2(\mathcal R(\varphi',\psi'))
$$
given by $[\sigma] {\rm Im}\delta \mapsto \widehat i_2([\sigma])$ is a bijection of finite sets. The result follows from \eqref{R-central}.\end{proof}

\begin{remark} The finiteness of $R(\varphi',\psi')$ can be determined in another way, namely, $R(\varphi',\psi')<\infty$ iff ${\rm rk}({\rm Im}(\varphi'-\psi'))={\rm rk}(A_2)$.\end{remark}

\subsection{Abelianization}\label{Abelian}
In the case where $R(\varphi',\psi')=\infty$ in the situation above, one can make use of the abelianization of the $G_j$'s as follows. Again, we have the following commutative diagram

\begin{equation}\label{exact1}
\begin{CD}
    1 @>>> H_1    @>{i_1}>>  G_1 @>{p_1}>>    Q_1 @>>> 1 \\
    @.     @V{\varphi'}V{\psi'}V  @V{\varphi}V{\psi}V   @V{\overline \varphi}V{\overline \psi}V @.\\
    1 @>>> H_2    @>{i_2}>>  G_2 @>{p_2}>>    Q_2 @>>> 1 
 \end{CD}
\end{equation}
where $H_i=[G_i,G_i]$ for $i=1,2$, $\varphi' = \varphi|_{H_1}$, $\psi' = \psi|_{H_1}, Q_1=G_1^{ab}, Q_2=G_2^{ab}$.

At the space level, the abelianization $G_j^{ab}$ is the fundamental group of some torus $T_j$, $i=1,2$. Since $R(\varphi,\psi)<\infty$, $R(\overline \varphi, \overline \psi)<\infty$. This forces (see e.g. \cite{w1} or \cite{daci-peter}) $\dim T_1\ge \dim T_2$, i.e., ${\rm rk}(Q_1) \ge {\rm rk}(Q_2)$. Following the argument from \cite[Lemma 3]{w1}, we can further assume that $\overline \varphi$ and $\overline \psi$ factor through another quotient $\overline {Q_1}$ of $Q_1$ such that 
${\rm rk}(\overline {Q_1})={\rm rk}(Q_2)$. Thus \eqref{exact1} becomes

\begin{equation}\label{exact2}
\begin{CD}
    1 @>>> K_1    @>{\hat i_1}>>  G_1 @>{\hat p_1}>>    \overline {Q_1} @>>> 1 \\
    @.     @V{\varphi'_1}V{\psi'_1}V  @V{\varphi}V{\psi}V   @V{\overline \varphi_1}V{\overline \psi_1}V @.\\
    1 @>>> H_2    @>{i_2}>>  G_2 @>{p_2}>>    Q_2 @>>> 1 
 \end{CD}
\end{equation}

Now, ${\rm rk}(\overline {Q_1})={\rm rk}(Q_2)$ and $R(\overline \varphi_1, \overline \psi_1)<\infty$, it follows that $Coin(\overline \varphi_1, \overline \psi_1)=1$. Thus the Reidemeister classes $\mathcal R(\varphi'_1,\psi'_1)$ inject into $\mathcal R(\varphi,\psi)$. Thus, it follows from \cite[Corollary 1]{w1} or \cite[Theorem 2.1]{daci-peter6} that
\begin{equation}\label{addition}
R(\varphi,\psi)=\sum_{[\overline \alpha]\in \mathcal R(\overline \varphi_1, \overline \psi_1)} R(\tau_{\alpha} \varphi'_1, \psi'_1).
\end{equation}
Now, $\varphi'_1,\psi'_1$ are homomorphisms induced by maps between compact nilmanifolds where the dimension of the codomain is smaller than that of $N$. After a finite number of steps, we arrive at computing $R(\tau_{\alpha}\varphi'_1,\psi'_1)$ where the target group is the fundamental group of a torus, hence such Reidemeister numbers are simply the cokernels of the difference of the two homomorphisms. In this sense, the Reidemeister number of a pair of maps from $X$ to a nilmanifold $N$ can be computed using the formula \eqref{addition} together with the arguments described above.

Using the arguments as in the previous subsections \ref{Central} and \ref{Abelian}, one can compute $R(f,g)$ when the codomain is a compact nilmanifold.

We end this section with an example illustrating the computation discussed here.

\begin{example}\label{T_to_N}
Consider the following finitely generated torsion-free nilpotent groups 
\begin{equation*}
\begin{aligned}
G_1=\langle a,b,c,d,e,t \mid &[a,b]=c, [a,d]=e, \\
                                             &[a,c]=[a,e]=[a,t]=[b,c] \\
                                             =&[b,d]=[b,t]=[c,d]=[c,t] \\
                                             =&[d,e]=[d,t]=[e,t]=1\rangle
\end{aligned}
\end{equation*}
and
$$
G_2=\langle \alpha, \beta, \gamma \mid [\alpha, \beta]=\gamma, [\alpha,\gamma]= [\beta,\gamma]=1\rangle.
$$
Note that, using Hall's identity, $[b,e]=[c,e]=1$ in $G_1$.

Now, $[G_1,G_1]=\langle c,e \mid [c,e]=1 \rangle$ is central and $[G_2,G_2]=\langle \gamma \rangle$ is the center of $G_2$. Moreover, $G_1^{ab}=\langle \bar a, \bar b, \bar d, \bar t\rangle \cong \mathbb Z^4$ and $G_2^{ab}=\langle \bar \alpha, \bar \beta \rangle \cong \mathbb Z^2$.

For $i=1,2,3$, we define homomorphisms $\varphi_i:G_1\to G_2$ by
\begin{equation*}
\begin{aligned}
\varphi_1: &a\mapsto \alpha^2; b\mapsto \beta; c\mapsto \gamma^2; d\mapsto 1; e \mapsto 1; t \mapsto 1 \\
\varphi_2: &a\mapsto \alpha; b\mapsto 1; c\mapsto 1; d \mapsto 1; e \mapsto 1; t\mapsto \alpha; \\
\varphi_3: &a\mapsto \beta; b\mapsto \alpha; c\mapsto \gamma^{-1}; d\mapsto \alpha; e\mapsto \gamma^{-1}; t \mapsto 1.
\end{aligned}
\end{equation*}

Consider the maps $F=(\varphi_1, \varphi_1)$ and $G=(\varphi_2, \varphi_3)$ and the commutative diagram
\begin{equation}\label{central-exact1}
\begin{CD}
    1 @>>> A_1    @>{i_1}>>  G_1 @>{p_1}>>    B_1 @>>> 1 \\
    @.     @V{F'}V{G'}V  @V{F}V{G}V   @V{\overline F}V{\overline G}V @.\\
    1 @>>> A_2\times A_2    @>{i_2\times i_2}>>  G_2\times G_2 @>{p_2\times p_2}>>    B_2\times B_2 @>>> 1 
 \end{CD}
\end{equation}
where $A_i=[G_i,G_i]$ and $B_i=G_i^{ab}$, $i=1,2$. Now,
\begin{equation*}
\begin{aligned}
F=(\varphi_1, \varphi_1): &a\mapsto (\alpha^2, \alpha^2); b\mapsto (\beta, \beta); c\mapsto (\gamma^2,\gamma^2); d\mapsto (1,1); e \mapsto (1,1); t \mapsto (1,1) \\
G=(\varphi_2,\varphi_3): &a\mapsto (\alpha, \beta); b\mapsto (1,\alpha); c\mapsto (1,\gamma^{-1}); d \mapsto (1,\alpha); e \mapsto (1,\gamma^{-1}); t\mapsto (\alpha,1). 
\end{aligned}
\end{equation*}

Similarly, we have
\begin{equation*}
\begin{aligned}
\overline F=(\overline \varphi_1, \overline \varphi_1): &\bar a\mapsto (\bar \alpha^2, \bar \alpha^2); \bar b\mapsto (\bar \beta, \bar \beta); \bar d\mapsto (\bar 1,\bar 1); \bar t \mapsto (\bar 1,\bar 1) \\
\overline G=(\overline \varphi_2,\overline \varphi_3): &\bar a\mapsto (\bar \alpha,\bar \beta); \bar b\mapsto (\bar 1,\bar \alpha); \bar d \mapsto (\bar 1,\bar \alpha); \bar t\mapsto (\bar \alpha,\bar 1). 
\end{aligned}
\end{equation*}

It follows that
$$
\overline F=\begin{pmatrix}
                        2 & 0 & 0 & 0 \\
                        0 & 1 & 0 & 0 \\
                        2 & 0 & 0 & 0 \\
                        0 & 1 & 0 & 0 
                        \end{pmatrix} \qquad \text{and} \qquad \overline G=\begin{pmatrix}
                        											1 & 0 & 0 & 1 \\
                        											0 & 0 & 0 & 0 \\
                        											0 & 1 & 1 & 0 \\
                        											1 & 0 & 0 & 0 
                        											\end{pmatrix}
$$	
so that 
$$\overline F-\overline G=\begin{pmatrix}
                        					1 & 0 & 0 & -1 \\
                        					0 & 1 & 0 & 0 \\
                        					2 & -1 & -1 & 0 \\
                        					-1 & 1 & 0 & 0 
                        					\end{pmatrix}$$
					and 
					$$R(\overline F, \overline G)=|\det (\overline F - \overline G)|=1.$$
Now, we have $F': c\mapsto (\gamma^2, \gamma^2); e\mapsto (1,1)$ and $G': c\mapsto (1, \gamma^{-1}); e\mapsto (1,\gamma^{-1})$. Thus,
$$
F'=\begin{pmatrix}
                        2 & 0  \\
                        2 & 0 
                        \end{pmatrix} \qquad \text{and} \qquad G'=\begin{pmatrix}
                                            								0 & 0  \\
                        											-1 & -1 \end{pmatrix}
$$	
so that 
$$F'-G'=\begin{pmatrix}
                       2 & 0 \\
                       3 & 1  \end{pmatrix}$$
					and 
					$$R(F', G')=|\det (F' - G')|=2.$$
To complete the calculation, we need to find ${\rm Im}\delta$ where $\delta: Coin(\overline F, \overline G) \to \mathcal R(F',G')$. Note that $\overline F-\overline G$ is invertible so the kernel, which is the same as $Coin(\overline F, \overline G)$, must be trivial. Hence $|{\rm Im}\delta|=1$. It follows from Lemma \ref{central-lemma} that $R(F,G)=2$ and hence $R(\varphi_1,\varphi_2,\varphi_3)=2$.														
\end{example}

\section{Concluding Remarks}
Let $f,g:M\to N$ be maps between closed orientable manifolds where $\dim M\ge \dim N$. When $N$ is a Jiang type space as in Theorem \ref{jiang-type}, one can use the coincidence theory for multiple maps to give an upper bound for $N(f,g)$ as follows.

Suppose $m=\dim M, n=\dim N$ and $m=(k-1)n$ for some positive integer $k\ge 2$. Let $F,G: M\to N^{(k-1)}$ where $F=(f,f,...,f)$ and $G=(g,g,...,g)$. Now $F,G$ are maps between two closed orientable manifolds of the same dimension so that $L(F,G)$ is the usual homological Lefschetz coincidence trace. Then we have the following result.

\begin{theorem}\label{jiang-type-estimate}
Let $f,g:M\to N$ be maps between two closed orientable manifolds where $N$ is a Jiang space; a nilmanifold; an orientable coset space $G/K$ of a compact connected Lie group $G$ by a closed subgroup $K$; or a $\mathcal C$-nilpotent space whose fundamental group has a finite index center where $\mathcal C$ is the class of finite groups. Suppose $\dim M=(k-1)n, \dim N=n$ and $k\ge 2$. Let $F=(\underbrace{f,f,...,f}_{(k-1)})$ and $G=(\underbrace{g,g,...,g}_{(k-1)})$. If $L(F,G)\ne 0$ then $\sqrt[k-1]{|L(F,G)|} \ge R(f,g)\ge N(f,g)$.
\end{theorem}
\begin{proof}
It follows from Theorem \ref{jiang-type} that $L(F,G)\ne 0 \Rightarrow N(F,G)=R(F,G)$. By the inequality \eqref{R-ineq} we have $R(F,G)=R(f,g,g,...,g)\ge (R(f,g))^{(k-1)}$. Since the Nielsen coincidence classes have coincidence index of the {\it same} sign, it follows that $|L(F,G)|\ge N(F,G)=R(F,G)$. The inequality follows.
\end{proof}

Next, we show that Theorem \ref{jiang-type-estimate} can be applied to give an upper bound for $N(f,g)$ in general.

\begin{lemma}\label{sphere-factors}
Let $f,g:M\to N$ be maps between closed connected orientable manifolds with $m=\dim M\ge \dim N=n$. Let $r\in \mathbb N\cup \{0\}$ such that $r+m=(k-1)n$ for some integer $k\ge 2$. Denote by $p:\tilde S^r\times M\to M$ be the canonical projection where $\tilde S^r$ is the $r$-sphere if $r>0$ and $\tilde S^0$ is a point. Let $\tilde f=f\circ p, \tilde g=g\circ p$. Then $R(f,g)=R(\tilde f, \tilde g)$.
\end{lemma}
\begin{proof} Suppose $\varphi, \psi$ and $\tilde \varphi, \tilde \psi$ denote the homomorphisms induced by $f,g$ and by $\tilde f,\tilde g$, respectively on the fundmental groups. Consider the commutative diagram
\begin{equation*}
\begin{CD}
    \pi_1(\tilde S^r)    @>>>  \pi_1(\tilde S^r)\times \pi_1(M) @>{p_{\#}}>>    \pi_1(M) \\
    @VVV  @V{\tilde \varphi}V{\tilde \psi}V   @V{\varphi}V{\psi}V \\
    \{1\}    @>>>  \pi_1(N) @>{=}>>    \pi_1(N) 
 \end{CD}
\end{equation*}

Since $\tilde \varphi(a,\sigma)=\varphi(\sigma)$ and $\tilde \psi(a,\sigma)=\psi(\sigma)$, it follows that the $\tilde \varphi$-$\tilde \psi$ twisted conjugacy classes in $\pi_1(N)$ coincide with the $\varphi$-$\psi$ twisted conjugacy classes in $\pi_1(N)$. Hence $R(f,g)=R(\tilde f,\tilde g)$.
\end{proof}

\begin{theorem}\label{jiang-type-estimate-general}
Let $f,g:M\to N$ be the same as in Theorem \ref{jiang-type-estimate} with no restrictions on $\dim M, \dim N$. Let $r\in \mathbb N\cup \{0\}$ such that $r+m=(k-1)n$ for some integer $k\ge 2$. Denote by $p:\tilde S^r\times M\to M$ be the canonical projection where $\tilde S^r$ is the $r$-sphere if $r>0$ and $\tilde S^0$ is a point. Let $\tilde f=f\circ p, \tilde g=g\circ p$.
Let $F=(\underbrace{\tilde f,\tilde f,...,\tilde f}_{(k-1)})$ and $G=(\underbrace{\tilde g,\tilde g,...,\tilde g}_{(k-1)})$ be maps from $\tilde S^r\times M \to N^{(k-1)}$. If $L(F,G)\ne 0$ then $\sqrt[k-1]{|L(F,G)|} \ge R(f,g)\ge N(f,g)$.
\end{theorem}
\begin{proof} The proof is a straightforward application of Lemma \ref{sphere-factors} and of Theorem \ref{jiang-type-estimate}.
\end{proof}

\end{document}